\documentclass[11pt,a4paper,twoside]{article}
\usepackage{amssymb}
\usepackage{amsfonts}
\usepackage{geometry}
\usepackage{color}
\usepackage{mathtools}
\usepackage{amsmath,amsthm,amssymb,amsfonts}
\usepackage{CJK}
\usepackage{latexsym}
\usepackage{graphicx}
\usepackage{pgfpages}
\usepackage{mathrsfs}
\usepackage{ytableau}
\usepackage{enumerate}
\usepackage{comment}
\usepackage{cite}
\usepackage[title,titletoc]{appendix}
\usepackage[symbols,nogroupskip,sort=none]{glossaries-extra}
\usepackage[hyphens]{url}
\usepackage[bookmarks=false,hidelinks, colorlinks, citecolor=red,linkcolor=blue]{hyperref}
\usepackage[affil-it]{authblk}
\allowdisplaybreaks

\newtheorem{defn}{Definition}[section]
\newtheorem{theorem}[defn]{Theorem}
\newtheorem{lemma}[defn]{Lemma}
\newtheorem{proposition}[defn]{Proposition}
\newtheorem{corollary}[defn]{Corollary}
\newtheorem{remark}[defn]{Remark}

\newcommand{\card}{\operatorname{card}}

\begin{document}
	
		\title{Symbolic Computations of the Two-Colored Diagrams for Central  Configurations of the Planar N-vortex Problem}
	\author [1] {Xiang Yu}
	\author [2] {Shuqiang Zhu}
	\affil[1] {Center for Applied Mathematics and KL-AAGDM, Tianjin University, Tianjin,  300072, China}
	\affil[2] {School of Mathematics, Southwestern 	University of Finance and Economics, Chengdu 611130, China}
	\affil[ ]{xiang.zhiy@foxmail.com, yuxiang\_math@tju.edu.cn, zhusq@swufe.edu.cn}
	\date{}
	
	\maketitle

\begin{abstract}
	We apply the  singular sequence method to investigate   the finiteness problem for stationary configurations of the planar N-vortex problem. 
	The initial step of the singular sequence method involves identifying all two-colored diagrams. These diagrams represent potential scenarios where finiteness may fail. We develop a  symbolic computation algorithm to
	determine all two-colored diagrams for central  configurations of the planar N-vortex problem.
\end{abstract}

\textbf{Keywords:}: {   Point vortices;  central configuration;  finiteness; singular sequence; symbolic computation }.

\section{Introduction}

The \emph{planar $N$-vortex problem} which originated from Helmholtz's work in 1858 \cite{helmholtz1858integrale}, considers the motion of point vortices in a fluid plane. It was given a Hamiltonian formulation by Kirchhoff as follows:
\[
\Gamma_n \dot{\mathbf{r}}_n = J \frac{\partial H}{\partial \mathbf{r}_n} = J \sum_{1 \leq j \leq N, j \ne n} \Gamma_j \Gamma_n \frac{\mathbf{r}_j - \mathbf{r}_n}{|\mathbf{r}_j - \mathbf{r}_n|^2}, \quad n = 1, \ldots, N.
\]
Here, $J = \begin{bmatrix} 0 & 1 \\ -1 & 0 \end{bmatrix}$, 
$\mathbf{r}_n = (x_n, y_n) \in \mathbb{R}^2$, and $\Gamma_n$ $(n = 1, \ldots, N)$ are the positions and vortex strengths (or vorticities) of the vortices, and 
the Hamiltonian is $H = -\sum_{1 \leq j < k \leq N} \Gamma_j \Gamma_k \ln|\mathbf{r}_j - \mathbf{r}_k|$, where $|\cdot|$ denotes the Euclidean norm in $\mathbb{R}^2$. The $N$-vortex problem is a widely used model for providing finite-dimensional approximations to vorticity evolution in fluid dynamics, especially when the focus is on the trajectories of the vorticity centers rather than the internal structure of the vorticity distribution \cite{Newton2001}.

An interesting set of special solutions of the dynamical system are homographic solutions, where the relative shape of the configuration remains constant during the motion. An excellent review of these solutions can be found in \cite{aref2003vortex, Newton2001}. 
Following O'Neil, we refer to the corresponding configurations as \emph{stationary}. The only stationary configurations are equilibria, rigidly translating configurations (where the vortices move with a common velocity), relative equilibria (where the vortices rotate uniformly), and collapse configurations (where the vortices collide in finite time) \cite{o1987stationary}.

 The equations governing stationary configurations are similar to those describing central configurations in celestial mechanics. 
Albouy and Kaloshin introduced a novel method to study the finiteness of five-body central configurations in celestial mechanics \cite{Albouy2012Finiteness}. The first author successfully extended this approach to fluid mechanics. Using this new method, the first author established not only the finiteness of four-vortex relative equilibria for any four nonzero vorticities but also the finiteness of four-vortex collapse configurations for a fixed angular velocity. This represents the first result on the finiteness of collapse configurations for $N \geq 4$ \cite{hampton2009finiteness, yu2021Finiteness}.

In this paper, we focus on the finiteness of N-vortex relative equilibria and collapse configurations.   We apply the singular sequence method developed by the first author in \cite{yu2021Finiteness}. 	The initial step of the singular sequence method involves identifying all two-colored diagrams. These diagrams represent potential scenarios where finiteness may fail. The singular sequence method requires substantial computations for $N \geq 5$. Chang and Chen developed symbolic computation algorithms to implement this approach for the finiteness of central configurations in celestial mechanics \cite{2023Chen-1, 2023Chen-2}. With these algorithms, Chang and Chen made significant progress on the finiteness of six-body central configurations. We adapt their algorithm for the determination of the two-colored diagrams in the vortex problem.

The paper is structured as follows. In Sect. \ref{sec:basicnotations}, we introduce notations and definitions. In Sect. \ref{sec:pri},  we   briefly review the singular sequence method and the two-colored diagrams. In Sect. \ref{sec:moreproperty}, we identify constraints when some particular sub-diagrams appear. In Sect. \ref{sec:matrixrules}, we transform diagram rules into matrix rules for the $zw$-matrices and discuss the outline of the algorithm for the $zw$-matrices. 
In Sect. \ref{sec:diagram&constraints}, we sketch those diagrams for the five-vortex problem.

\section{Basic notations}\label{sec:basicnotations}
\indent\par
We recall some basic notations on stationary configurations and direct readers to a more comprehensive introduction provided by O'Neil \cite{o1987stationary} and Yu \cite{yu2021Finiteness}.

We represent vortex positions $\mathbf{r}_n \in \mathbb{R}^2$ as complex numbers $z_n \in \mathbb{C}$. The equations of motion are  $\dot{z}_n = \textbf{i}V_n$, where
\begin{equation}\label{vectorfield}
	V_n= \sum_{1 \leq j \leq N, j \neq n} \frac{\Gamma_j z_{jn}}{r_{jn}^2}= \sum_{ j \neq n} \frac{\Gamma_j }{{\overline{z}_{jn}}}.
\end{equation}
Here, $z_{jn} = z_n - z_j$, $r_{jn} = |z_{jn}| = \sqrt{z_{jn}{\overline{z}_{jn}}}$, $\textbf{i} = \sqrt{-1}$, and the overbar denotes complex conjugation.

Let $\mathbb{C}^N= \{ z = (z_1,  \ldots, z_N):z_j \in \mathbb{C}, j = 1,  \ldots, N \}$ denote the space of configurations for $N$ point vortex. The collision set is defined as  $\Delta=\{ z \in \mathbb{C}^N:z_j=z_k ~~\emph{for some}~~ j\neq k  \} $. The space of collision-free configurations is given by $\mathbb{C}^N \backslash \Delta$.

\begin{defn}\label{def:LI}
	The following quantities and notations are defined:
	\begin{center}
		$\begin{array}{cc}
		\text{Total vorticity} & \Gamma =\sum_{j=1}^{N}\Gamma_j  \\
		\text{Total vortex angular momentum} & L =\sum_{1\leq j<k\leq N}\Gamma_j\Gamma_k.  
		\end{array}$
	\end{center}
	For $J=\{j_1, ..., j_n\}\subset \{1, ..., N\}$, we also define 
	\[ \Gamma_J=\Gamma_{j_1, ..., j_n}=\sum_{j\in J} \Gamma_j, \ L_J=L_{j_1, ..., j_n}=\sum_{j<k, j,k \in J} \Gamma_j \Gamma_k.\] 
\end{defn}

 A motion is called homographic if the relative shape remains constant. 
	Following  O'Neil \cite{o1987stationary}, we term  a corresponding  configuration as a  \emph{stationary configuration}.  Equivalently, 
\begin{defn} \label{def-1}
	A configuration $z \in \mathbb{C}^N \backslash \Delta$ is stationary if there exists a constant $\Lambda\in {\mathbb{C}}$ such that
	\begin{equation}\label{stationaryconfiguration}
		V_j-V_k=\Lambda(z_j-z_k), ~~~~~~~~~~ 1\leq j, k\leq N.
	\end{equation}
\end{defn}

 There are  only four kinds of homographic motions,  equilibria, translating with a common velocity, uniformly rotating, and   homographic motions that collapse in finite time.   
	Following \cite{o1987stationary, hampton2009finiteness, yu2021Finiteness},  we term 
	the  stationary
	configurations  corresponding to these four classes of  homographic motions as
	equilibria, rigidly translating configurations, relative equilibria  and collapse configurations. 
	Equivalently,

\begin{defn}\label{def-2}
	\begin{itemize}
		\item[i.] $z \in \mathbb{C}^N \backslash \Delta$ is an \emph{equilibrium} if $V_1=\cdots=V_N=0$.
		\item[ii.] $z \in \mathbb{C}^N \backslash \Delta$ is \emph{rigidly translating} if $V_1=\cdots=V_N=c$ for some $c\in \mathbb{C}\backslash\{0\}$.
		\item[iii.] $z \in \mathbb{C}^N \backslash \Delta$ is  a \emph{relative equilibrium} if there exist constants $\lambda\in \mathbb{R}\backslash\{0\},z_0\in \mathbb{C}$ such that $V_n=\lambda(z_n-z_0),~~~~~~~~~~ 1\leq n\leq N$.
		\item[iv.] $z \in \mathbb{C}^N \backslash \Delta$ is a \emph{collapse configuration} if there exist constants $\Lambda,z_0\in \mathbb{C}$ with $\emph{Im}(\Lambda)\neq0$ such that $V_n=\Lambda(z_n-z_0),~~~~~~~~~~ 1\leq n\leq N$.
	\end{itemize}
\end{defn}

\begin{defn}
	A configuration $z$ is equivalent to $z'$ if there exist $a, b \in \mathbb{C}$ with $b \neq 0$ such that $z'_n = b(z_n + a)$ for $1 \leq n \leq N$.

		A configuration is called translation-normalized if its translation freedom is removed, rotation-normalized if its rotation freedom is removed, and dilation-normalized if its dilation freedom is removed. 
		A configuration normalized in translation, rotation, and dilation is termed a \textbf{normalized configuration}. 
\end{defn}

	We count the stationary configurations according to the equivalence classes. Counting equivalence classes is the same as counting normalized configurations. Note that the removal of any of these three freedoms can be performed in various  ways.

\section{Singular sequences for central configurations and coloring rules}\label{sec:pri} 
\label{Preliminaries}

In this section, we briefly review   the basic elements of the Albouy-Kaloshin approach developed   by Yu  \cite{yu2021Finiteness} for the finiteness of relative equilibria and collapse configurations, including,  among others, the notation of central configurations, the extended system, the notation of singular sequences, the two-colored diagrams, and the rules for  the two-colored diagrams. For  a more comprehensive introduction, please refer to   \cite{yu2021Finiteness}.

\subsection{Central configurations of the planar N-vortex problem}
Recall Definition \ref{def-2}. Equations of relative equilibria and collapse configurations share the form:
\begin{equation}\label{stationaryconfiguration1}
	V_n=\Lambda(z_n-z_0),~~~~~~~~~~ 1\leq n\leq N,
\end{equation}
where $\Lambda\in \mathbb{R}\backslash\{0\}$  indicates relative equilibria and $\Lambda\in \mathbb{C}\backslash\mathbb{R}$  indicates  collapse configurations.
\begin{defn}\label{def:cc}
	Relative equilibria and collapse configurations are both called \textbf{central configurations}.
\end{defn}

The equations (\ref{stationaryconfiguration1}) read 
\begin{equation}\label{stationaryconfiguration2}
	\Lambda z_n= V_n,~~~~~~~~~~ 1\leq n\leq N,
\end{equation}
if the translation freedom is removed, i.e., we substitute $z_n$ with $z_n + z_0$ in equations (\ref{stationaryconfiguration2}). The solutions then satisfy:
\begin{equation}\label{center0}
	M=0, \ \Lambda I= L.
\end{equation}
To remove dilation freedom, we enforce $|\Lambda| = 1$.

Introduce a new set of variables $w_n$ and a ``conjugate"
relation:
\begin{equation}\label{stationaryconfiguration3}
	\Lambda z_n=\sum_{ j \neq n} \frac{\Gamma_j }{{w_{jn}}},\ \ 
	\overline{\Lambda} w_n=\sum_{ j \neq n} \frac{\Gamma_j }{{z_{jn}}},\ \ \  1\leq n\leq N,
\end{equation}
where $z_{jn}=z_{n}-z_{j}$ and $w_{jn}=w_{n}-w_{j}$.

The rotation symmetry of \eqref{stationaryconfiguration2} leads to   the invariance of
(\ref{stationaryconfiguration3}) under the map 
\[R_a: (z_1, ..., z_n, w_1, ..., w_n) \mapsto (az_1, ..., az_N, a^{-1} w_1, ..., a^{-1}w_N)\]
for any $a\in\mathbb{C}\backslash \{0\}$. 

Introduce  the variables $Z_{jk},W_{jk}\in \mathbb{C}$ $(1\leq j< k\leq N)$ such that
$Z_{jk}=1/w_{jk}, W_{jk}=1/z_{jk}$. For $1\leq k< j\leq N$ we set $Z_{jk}=-Z_{kj}, W_{jk}=-W_{kj}$. Then equations (\ref{stationaryconfiguration2}) together with the condition $z_{12}\in\mathbb{R}$ and $|\Lambda|=1$ are embedded into the following extended system
\begin{equation}\label{equ:complexcc}
	\begin{array}{cc}
		\Lambda z_n=\sum_{ j \neq n} \Gamma_j Z_{jn},&1\leq n\leq N, \\
		\overline{\Lambda} w_n=\Lambda^{-1} w_n=\sum_{ j \neq n} \Gamma_j W_{jn},& 1\leq n\leq N, \\
		Z_{jk} w_{jk}=1,&1\leq j< k\leq N, \\
		W_{jk} z_{jk}=1,&1\leq j< k\leq N, \\
		z_{jk}=z_k-z_j,~~~  w_{jk}=w_k-w_j,&1\leq j, k\leq N, \\
		Z_{jk}=-Z_{kj},~~~ W_{jk}=-W_{kj},&1\leq k< j\leq N, \\
		z_{12}=w_{12}.
	\end{array}
\end{equation}
This is a polynomial system in the  variables $\mathcal{Q}=(\mathcal{Z},\mathcal{W})\in\mathbb{C}^{2\mathfrak{N}}$, here
\begin{center}
	$\mathcal{Z}=(\mathcal{Z}_{1},\mathcal{Z}_{2},\ldots,\mathcal{Z}_{\mathfrak{N}})=(z_1,z_2,\ldots,z_N,Z_{12},Z_{13},\ldots,Z_{(N-1)N})$,
	$\mathcal{W}=(\mathcal{W}_{1},\mathcal{W}_{2},\ldots,\mathcal{W}_{\mathfrak{N}})=(w_1,w_2,\ldots,w_N,W_{12},W_{13},\ldots,W_{(N-1)N})$.
\end{center}
and $\mathfrak{N}=N(N+1)/2$.  

\begin{defn}\label{def:positivenormalizedcentralconfiguration}
 A \emph{complex normalized} central configuration of the planar
		$N$-vortex problem is a solution of (\ref{equ:complexcc}).  A \emph{real  normalized} central configuration of the planar
		$N$-vortex problem is a  complex normalized central configuration  satisfying $z_n={\overline{w}}_n$ for any $n=1, \ldots, N$.  
\end{defn}

 Note that a real  normalized central configuration of Definition \ref{def:positivenormalizedcentralconfiguration} is exactly  a central configuration of 
	Definition \ref{def:cc}.  	
We will use the name ``distance" for the $r_{jk}=\sqrt{z_{jk}{w_{jk}}}$. Strictly speaking, the distances $r_{jk}=\sqrt{z_{jk}{w_{jk}}}$ are now bi-valued. However, only the squared distances appear in the system, so we shall  understand $r^2_{jk}$ as $z_{jk}w_{jk}$ from now on.

\subsection{Singular sequences}
Let $\|\mathcal{Z}\|=\max_{j=1,2,\ldots,\mathfrak{N}}|\mathcal{Z}_{j}|$ be the modulus of the maximal component of
the vector $\mathcal{Z}\in \mathbb{C}^\mathfrak{N}$. Similarly, set  $\|\mathcal{W}\|=\max_{k=1,2,\ldots,\mathfrak{N}}|\mathcal{W}_{k}|$.

One important feature of System \eqref{equ:complexcc} is the symmetry: 
if $\mathcal Z, \mathcal W$ is a solution, so is $ a\mathcal Z, a^{-1} \mathcal W$ for any $a\in\mathbb{C}\backslash\{0\}$. Thus, we can replace 
the  normalization $z_{12}=w_{12}$  in System \eqref{equ:complexcc}  by 
$\|\mathcal{Z}\|=\|\mathcal{W}\|$.  From now on, we consider System \eqref{equ:complexcc} with this new normalization.

Consider
a sequence $\mathcal{Q}^{(n)}$, $n=1,2,\ldots$, of solutions  of (\ref{equ:complexcc}).  Take a sub-sequence such that the maximal component of $\mathcal{Z}^{(n)}$ is fixed, i.e., there is a $j\in \{1,2,\ldots,\mathcal{N}\}$ that is  independent 
of $n$ such that  $\|\mathcal{Z}^{(n)}\|=|\mathcal{Z}^{(n)}_{j}|$. 
Extract again in such a way that the sequence $\mathcal{Z}^{(n)}/\|\mathcal{Z}^{(n)}\|$ converges.
Extract again  in such a way that  the maximal component of $\mathcal{W}^{(n)}$ is fixed. Finally,  extract  in
such a way that the sequence $\mathcal{W}^{(n)}/\|\mathcal{W}^{(n)}\|$ converges.

\begin{defn}[Singular sequence]
	Consider a sequence of complex normalized central configurations with the property  that $\mathcal{Z}^{(n)}$ is unbounded. A
	sub-sequence extracted by the above process is called
	a \emph{singular sequence}.
\end{defn}

\begin{lemma}\label{Eliminationtheory}\cite{Albouy2012Finiteness}
	Let $\mathcal{X}$ be a closed algebraic subset of $\mathbb{C}^m$ and $f:\mathbb{C}^m\rightarrow \mathbb{C}$ be a
	polynomial. Either the image $F(\mathcal{X})\subset\mathbb{ C}$ is a finite set, or it is the complement
	of a finite set. In the second case one says that f is dominating.
\end{lemma}

\subsection{The  two-colored diagrams} \label{sec:rule}

For two sequences of non-zero numbers, $a, b$, we use $a\sim b$,  $a\prec b$, $ a\preceq b$,  and $ a \approx b$ to represent ``$a/b\rightarrow 1$'', ``$a/b\rightarrow 0$'', ``$a/b$ is bounded'' and ``$a\preceq b$, $a\succeq  b$'' respectively.  

Recall that  a  singular sequence satisfy the property   $\|\mathcal{Z}^{(n)}\|=\|\mathcal{W}^{(n)}\|\to \infty$.  Set $\|\mathcal{Z}^{(n)}\|=\|\mathcal{W}^{(n)}\|=1/\epsilon^2$. Then $\epsilon\rightarrow 0$.  Following Albouy-Kaloshin, \cite{Albouy2012Finiteness}, the \emph{two-colored diagram} was introduced in \cite{yu2021Finiteness} to  classify the singular sequences.  Given a singular sequence, the indices of the vertices  will be written down. 
The first color, called the $z$-color (red),   is used to mark the maximal order components of $\mathcal{Z}$. If $z_k\approx \epsilon^{-2}$,  draw a $z$-circle around the
vertex $\textbf{k}$; If   $Z_{jk}\approx \epsilon^{-2}$,  draw a $z$-stroke between vertices $\textbf{k}$ and $\textbf{j}$. 
They constitute  the $z$-diagram. 
The second color, called the $w$-color (blue and dashed),   is used to mark the maximal order components of $\mathcal W$ in similar manner. Then we also have the  $w$-diagram. The two-colored diagram is the combination of the  $z$-diagram and the $w$-diagram,  see Figure \ref{fig:edges}.

If there
is either a $z$-stroke, or a $w$-stroke, or both between vertex $\textbf{k}$ and vertex $\textbf{l}$, we say that there is an edge between them.  There are three types of edges,
$z$-edges, $w$-edges and $zw$-edges, see Figure \ref{fig:edges}. 

\begin{figure}[h!]
	\centering
	\includegraphics[width=0.6\textwidth]{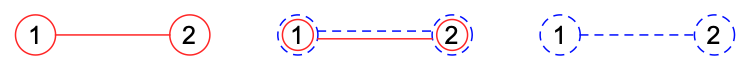} 
	\caption{On the left, vertices \textbf{1,2} are $z$-circled, and a $z$-edge is between them; In the middle, vertices \textbf{1,2} are $z$- and $w$-circled, and a $zw$-edge is between them; On the right,  vertices \textbf{1,2} are $w$-circled, and a $w$-edge is between them; 
	}
	\label{fig:edges}
\end{figure}

The following concepts were introduced to characterize some features of singular sequences.    In the $z$-diagram,   vertices  $\textbf{k}$ and $\textbf{l}$
are  called \emph{$z$-close},  if $z_{kl} \prec\epsilon^{-2}$; 
a $z$-stroke between  vertices   $\textbf{k}$ and  $\textbf{l}$ is  called  	a   \emph{maximal $z$-stroke} if $z_{kl} \approx  \epsilon^{-2}$; 
a subset of vertices are called \emph{an isolated component of the $z$-diagram} if there is no $z$-stroke  between a vertex of this subset and  a vertex of its complement. 
These concepts also apply to the $w$-diagram.

\begin{proposition}[Estimate]\label{Estimate1}\cite{yu2021Finiteness}
	For any $(k,l)$, $1\leq k<l\leq N$, we have $\epsilon^2\preceq z_{kl}\preceq \epsilon^{-2}$, $\epsilon^2\preceq w_{kl}\preceq \epsilon^{-2}$ and $\epsilon^2\preceq r_{kl}\preceq \epsilon^{-2}$.
	
	There is a $z$-stroke between $\textbf{k}$ and $\textbf{l}$ if and only if $w_{kl}\approx \epsilon^{2}$. Then $ r_{kl}\preceq 1$.
	
	There is a maximal $z$-stroke between $\textbf{k}$ and $\textbf{l}$ if and only if $z_{kl}\approx \epsilon^{-2}, w_{kl}\approx \epsilon^{2}$. Then $ r_{kl}\approx1$.
	
	There is a $z$-edge between $\textbf{k}$ and $\textbf{l}$ if and only if $z_{kl}\succ \epsilon^{2},w_{kl}\approx \epsilon^{2}$. Then $\epsilon^{2}\prec r_{kl}\preceq 1 $.
	
	There is a maximal $z$-edge between $\textbf{k}$ and $\textbf{l}$ if and only if $z_{kl}\approx \epsilon^{-2},w_{kl}\approx \epsilon^{2}$. Then $ r_{kl}\approx 1$.
	
	There is a $zw$-edge between $\textbf{k}$ and $\textbf{l}$ if and only if $z_{kl},w_{kl}\approx \epsilon^{2}$. This can be  characterized as $ r_{kl}\approx \epsilon^{2}$.
\end{proposition}

\begin{remark}
	By the estimates above, the strokes in a $zw$-edge are not maximal. A maximal $z$-stroke is exactly a maximal $z$-edge.
\end{remark}

The following rules for the two-colored diagrams  are valid if  ``$z$'' and ``$w$''  were switched.

\begin{description}
	\item[{Rule I}]
	There is something at each end of any $z$-stroke: another $z$-stroke
	or/and a $z$-circle drawn around the name of the vertex. A $z$-circle cannot be isolated; there must be a $z$-stroke emanating from it. There is at least one
	$z$-stroke in the $z$-diagram.
	\item[{Rule II}] If vertices $\textbf{k}$ and $\textbf{l}$
	are  $z$-close, they are both $z$-circled or both not
	$z$-circled.
	\item[{Rule III}]  The moment of vorticity of a set of vertices forming an isolated component of the $z$-diagram is $z$-close to the origin.
	\item[{Rule IV}]  Consider the $z$-diagram or an isolated component of it. If there
	is a $z$-circled vertex, there is another one.   If the $z$-circled vertices are all
	$z$-close together,  the total vorticity of these $z$-circled  vertices is zero.
	\item[{Rule V}]  There is at least one $z$-circle at certain end of any maximal $z$-stroke. As a result,
	if an isolated component of the $z$-diagram has no $z$-circled vertex,
	then it has no maximal $z$-stroke.
	\item[{Rule VI}]
	If there are two consecutive $z$-stroke, there is a third $z$-stroke closing the triangle.
\end{description}




\section{Constraints  when some sub-diagrams appear} \label{sec:moreproperty}

\indent\par
We collect some useful results in this section. 
	Recall that  $\Gamma_J, \Gamma_{j_1, ..., j_n}$, $L_J$ and $L_{j_1, ..., j_n}$  are defined in  Definition \ref{def:LI}.

\begin{proposition}\label{Prp:sumT12} \cite{yu2021Finiteness}
	Suppose that a diagram has two $z$-circled vertices (say $\textbf{1}$ and $\textbf{2}$) which are also $z$-close,   if none of all the other vertices is $z$-close with them,  then $\Gamma_1+\Gamma_2\neq 0$ and $\overline{\Lambda}z_{12}w_{12}\sim \frac{1}{\Gamma_1+\Gamma_2}$. In particular, vertices $\textbf{1}$ and $\textbf{2}$ cannot form a $z$-stroke.
\end{proposition}

\begin{corollary}\label{Cor:sumT12} 
	Suppose that a diagram has  two $z$-circled vertices (say $\textbf{1}$ and $\textbf{2}$) which also form a $z$-stroke. If  none of all the other vertices is $z$-close with them, then $z_{12}\approx \epsilon^{-2}$,  $\Gamma_1+\Gamma_2\neq 0$, and $w_1, w_2\preceq \epsilon ^2$.  
\end{corollary}

\begin{proof}
	If they are $z$-close, by Proposition \ref{Prp:sumT12}, they cannot form a $z$-stroke, which is a contradiction. Note that    Rule III implies  
	\[\epsilon^{-2}   \succ  \Gamma_1 z_1 +\Gamma_2 z_2 = (\Gamma_1+\Gamma_2) z_1 +\Gamma_2 (z_2-z_1).   \]
	We obtain $\Gamma_1+\Gamma_2\neq 0.$
	
	Note that $z_{1j}, z_{2j}\approx  \epsilon^{-2}, j\ge3$. Then 
	\[\bar{\Lambda}\sum_{j\ge 3} \Gamma_j w_j=\sum_{j\ge 3}\frac{\Gamma_1\Gamma_j}{z_{1j}}+\sum_{j\ge3}\frac{\Gamma_2\Gamma_j}{z_{2j}}\preceq \epsilon^{2}.\]
	By the equation $\sum_{j}\Gamma_j w_j=0$, we have 
	\[ \epsilon^{2}\succeq \Gamma_1 w_1+\Gamma_2w_2=(\Gamma_1+\Gamma_2)w_1+ \Gamma_2w_{21}.  \]
	Since $w_{21}\approx \epsilon^2$, we have $w_1, w_2\preceq \epsilon ^2$.  
\end{proof}

\begin{proposition}\label{Prp:LI} \cite{yu2021Finiteness}
	Suppose that  a fully $z$-stroked sub-diagram with  vertices $\{1,..., k\}, (k\ge 3)$ exists in isolation  in a diagram, and none of its vertices  is $z$-circled, then
	\begin{equation}\notag
		L_{1...k}= \sum_{i,j \in \{1, ..., k\}, i\ne j}\Gamma_i\Gamma_j=0.
	\end{equation}
	
\end{proposition}

\begin{corollary}\label{Cor:L&Gamma} 
	Suppose a fully \(z\)-stroked sub-diagram with vertices \(K = \{1, \ldots, k\}\), \(k \ge 3\), exists in isolation in a diagram, and none of its vertices is \(z\)-circled. 
	\begin{enumerate}
		\item If there is an isolated component \(I\) of the \(w\)-diagram such that \(K \subset I\), then the \(w\)-circled vertices in \(I\) cannot be exactly \(\{1, \ldots, k\}\). 
		\item Consider  any subset of \(K\) with cardinality \((k-1)\) , say \(K_1 = \{2, \ldots, k\}\). If there is an isolated component \(I\) of the \(w\)-diagram with \(K_1 \subset I\), then the \(w\)-circled vertices in \(I\) cannot be exactly \(K_1\). 
		\item If there is a vertex outside of \(K\), say \(k+1\), such that \(\{k+1\} \cup K\) forms an isolated component of the \(w\)-diagram and these \(k+1\) vertices are fully \(w\)-stroked, then there is at least one \(w\)-circle among them.  
		\item If there are several isolated components \(\{I_j, j = 1, \ldots, s\}\) of the \(w\)-diagram with \(K \subset \cup_{j=1}^s I_j\), then the \(w\)-circled vertices in \(\cup_{j=1}^s I_j\) cannot be exactly \(\{1, \ldots, k\}\). 
	\end{enumerate}
\end{corollary}

\begin{proof}
	First, we have \(L_{K} = 0\) by Proposition \ref{Prp:LI}, and the vertices of \(K\) are all \(w\)-close by  the estimate of Proposition \ref{Estimate1}.
	
	For part (1), if the \(w\)-circled vertices in \(I\) are exactly \(\{1, \ldots, k\}\), then by Rule IV, we have \(\sum_{i \in K} \Gamma_i = 0\). This leads to a contradiction because:
	\[
	\left(\sum_{i \in K} \Gamma_i\right)^2 = \sum_{i \in K} \Gamma_i^2 + 2L_{K}.
	\]
	The proof of part (4) is similar.
	
	For part (2), if the \(w\)-circled vertices in \(I\) are exactly \(K_1 = \{2, \ldots, k\}\), then by Rule IV, we have \(\sum_{i=2}^k \Gamma_i = 0\). Therefore:
	\[
	L_{K_1} = L_{K} - \Gamma_{1}\left(\sum_{i=2}^k \Gamma_i\right) = 0,
	\]
	which again leads to a contradiction since \(\sum_{i \in K_1} \Gamma_i = 0\).
	
	For part (3), if the component \(\{k+1\} \cup K\) is fully \(w\)-stroked but has no \(w\)-circle, then Rule IV implies that \(L_K = 0\) and \(L_K + \Gamma_{k+1} \sum_{i \in K} \Gamma_i = 0\). This leads to \(\sum_{i \in K} \Gamma_i = 0\), which is a contradiction.
\end{proof}

\begin{proposition}\label{Prp:isolate-z_12}
	Suppose that a diagram has an isolated \(z\)-stroke in the \(z\)-diagram, and its two ends are \(z\)-circled. Let \(\textbf{1}\) and \(\textbf{2}\) be the ends of this \(z\)-stroke. Suppose there is no other \(z\)-circle in the diagram. Then \(\Gamma_1 + \Gamma_2 \ne 0\), and \(z_{12}\) is maximal. The diagram forces \(\Lambda = \pm 1\) or \(\pm \textbf{i}\).
	Furthermore,
	\begin{itemize}
		\item If $\Lambda = \pm 1$, we have $\sum_{j=3}^N \Gamma_j=0$;
		\item 	If $\Lambda = \pm \textbf{i}$, we have $L=0$ and $\Gamma_1\Gamma_2=L_{3...N}$.
	\end{itemize}
\end{proposition}

\begin{proof}
	
	The facts  that	$\Gamma_1+\Gamma_2\neq 0$ and $z_{12}$ is maximal  follow from Corollary \ref{Cor:sumT12}. 
	Without loss of generality, assume $z_1\sim -\Gamma_2 a\epsilon^{-2}$ and $z_2\sim \Gamma_1 a \epsilon^{-2}$, then $$z_{12}\sim (\Gamma_1 +\Gamma_2) a \epsilon^{-2},\  \frac{1}{z_{2}}-\frac{1}{z_{1}}\sim (\frac{1}{\Gamma_1}+\frac{1}{\Gamma_2})\frac{\epsilon^2}{a}. $$
	
	The  System \eqref{equ:complexcc} yields
	\begin{equation}
		\begin{array}{c}
			\label{equ:iso-z12}\overline{\Lambda} w_{12}=
			(\Gamma_1+\Gamma_2)W_{12}+ \sum_{j=3}^N \Gamma_j (\frac{1}{z_{j2}}-\frac{1}{z_{j1}}) \cr 
			\Lambda z_{2} \sim \Gamma_1 Z_{12}.
		\end{array}{}
	\end{equation}
	
	The second equation of \eqref{equ:iso-z12} implies  $w_{12} \sim \frac{\epsilon^2}{ a\Lambda}$.  Note that
	$\frac{1}{z_{j2}}-\frac{1}{z_{j1}}\sim \frac{1}{z_{2}}-\frac{1}{z_{1}}$
	for all $j >2$ and that  $W_{12}=\frac{1}{z_{12}}$. The first equation of
	\eqref{equ:iso-z12} implies 
	\begin{equation}\label{equ:iso-z12-1}
		{\overline{\Lambda}}/{\Lambda}=1+ \sum_{j=3}^N \Gamma_j  (\frac{1}{\Gamma_1}+\frac{1}{\Gamma_2}).
	\end{equation}
	It follows that $\Lambda = \pm 1$ or $ \pm \textbf{i}$.
	
	If $\Lambda = \pm 1$, we have
	\[ 0= \sum_{j=3}^N \Gamma_j  (\frac{1}{\Gamma_1}+\frac{1}{\Gamma_2}),\  \Rightarrow\  \sum_{j=3}^N \Gamma_j =0. \]
	
	If $\Lambda = \pm \textbf{i}$, we obtain 
	\[ -2= \sum_{j=3}^N \Gamma_j  (\frac{1}{\Gamma_1}+\frac{1}{\Gamma_2}),\  L=0, \ \Rightarrow L=0, \ \Gamma_1 \Gamma_2=L_{3...N}.  \]
\end{proof}

Similarly, we have the following result.

\begin{proposition}\label{Prp:isolate-z_123}
	Suppose that a diagram has an isolated triangle of \(z\)-strokes in the \(z\)-diagram, where two of the vertices of the triangle are \(z\)-circled. Let \(\textbf{2}\) and \(\textbf{3}\) be the two \(z\)-circled vertices and \(\textbf{1}\) the other vertex. Suppose there is no other \(z\)-circle in the diagram. Then \(\Gamma_2 + \Gamma_3 \ne 0\), and \(z_{23}\) is maximal. The diagram forces \(\Lambda = \pm 1\) or \(\pm \textbf{i}\).  Furthermore,
	\begin{itemize}
		\item If $\Lambda = \pm 1$, we have $\sum_{j=4}^N \Gamma_j=0$;
		\item 	If $\Lambda = \pm \textbf{i}$, we have $L=0$ and $L_{123}=L_{4...N} + \Gamma_1( \sum _{j=4}^N \Gamma_j)$.
	\end{itemize}
\end{proposition}

\begin{proof}
	The facts that 	$\Gamma_2+\Gamma_3\ne 0$ and $z_{23}$ is maximal  follow from Corollary \ref{Cor:sumT12}.  Note that
	\[  \Gamma_2 z_2 +\Gamma_3 z_3 \prec \epsilon^{-2}, \ \Lambda  z_1 \sim \Gamma_2 Z_{21} +\Gamma_3 Z_{31} \prec \epsilon^{-2}.  \]
	Without loss of generality, assume
	\[ z_2\sim -\Gamma_3 a\epsilon^{-2}, \ z_3\sim \Gamma_2 a \epsilon^{-2}, \ Z_{21} \sim -\Gamma_3 b\epsilon^{-2}, \ Z_{31}\sim \Gamma_2 b \epsilon^{-2}.  \]
	Then
	\begin{eqnarray*}
		z_{23}\sim (\Gamma_2 +\Gamma_3) a \epsilon^{-2},\  \frac{1}{z_{3}}-\frac{1}{z_{2}}\sim (\frac{1}{\Gamma_2}+\frac{1}{\Gamma_3})\frac{\epsilon^2}{a}, \cr
		Z_{23}=\frac{1}{w_{23}}= \frac{1}{1/Z_{21} +1/Z_{13}} \sim -b \frac{\Gamma_2\Gamma_3}{(\Gamma_2+\Gamma_3)}  \epsilon^{-2}.
	\end{eqnarray*}
	Then similar to the above case, we have
	\begin{eqnarray*}
		\overline{\Lambda} w_{23}\sim
		(\Gamma_2+\Gamma_3)W_{23}+ \sum_{j\ne 2, 3}^N \Gamma_j (\frac{1}{z_{3}}-\frac{1}{z_{2}}) \cr 
		\Lambda z_{23} \sim  (\Gamma_2+\Gamma_3) Z_{23} +\Gamma_1 (Z_{13}-Z_{12}).
	\end{eqnarray*}
	
	Short computation reduces the two equations to
	\begin{eqnarray*}
		-\overline{\Lambda} \frac{a}{b}= \frac{\Gamma_2\Gamma_3}{\Gamma_2+\Gamma_3} (1+\sum_{j\ne 2, 3} \Gamma_j \frac{\Gamma_2+\Gamma_3} {\Gamma_2\Gamma_3}), \cr
		-\Lambda \frac{a}{b}= \frac{L_{123}}{\Gamma_2+\Gamma_3}.
	\end{eqnarray*}
	Then we obtain
	\[   \frac{\overline{\Lambda}}{\Lambda}  L_{123}= \Gamma_2\Gamma_3+ (\Gamma_2+\Gamma_3) \sum_{j\ne 2, 3} \Gamma_j.   \]
	It follows that $\Lambda = \pm 1$ or $ \pm \textbf{i}$.
	
	If $\Lambda = \pm 1$, we have  $\sum_{j=4}^N \Gamma_j=0$.
	If $\Lambda = \pm \textbf{i}$, we have $L=0$ and
	$$L=0, \  -L_{123}= \Gamma_2\Gamma_3+ (\Gamma_2+\Gamma_3) \sum_{j\ne 2, 3} \Gamma_j,$$
	which is equivalent to
	$L_{123}=L_{4...N} + \Gamma_1( \sum _{j=4}^N \Gamma_j), \ L=0. $
\end{proof}

\begin{proposition}\label{Prp:triangle} 
	Assume there is a triangle with vertices \(\textbf{1}, \textbf{2}, \textbf{3}\) that is fully \(z\)- and \(w\)-stroked, and fully \(z\)- and \(w\)-circled. Moreover, assume that  the triangle is isolated in the \(z\)-diagram. Then there must exist  some \(k > 3\) such that \(z_{k1} \preceq 1\).
\end{proposition}

\begin{proof}
	By Proposition \ref{Estimate1} and Rule IV, we have  \[z_1\sim z_2\sim z_3, \ w_1\sim w_2\sim w_3, \ \Gamma_{1}+\Gamma_{2}+\Gamma_{3}=0.\]
	Suppose that it holds  $z_{k1}\succ 1$ for all $k>3$. Then  $\frac{1}{z_{kj}}- \frac{1}{z_{k1}}=\frac{z_{1j}}{z_{kj}z_{k1}} \prec \epsilon^2$ for all $k>3, 1\le j\le 3$, and so
	\[\bar{\Lambda}\sum_{j=1}^{3}\Gamma_j w_j=\sum_{k\ge 4} \sum_{j=1}^{3}\frac{\Gamma_k\Gamma_j}{z_{kj}}=\sum_{k\ge 4} \sum_{j=1}^{3}\Gamma_k\Gamma_j  (\frac{1}{z_{k1}}+\frac{1}{z_{kj}}- \frac{1}{z_{k1}} ) \prec \epsilon^{2}.\]

	By the fact that  $w_{12}, w_{13}, w_{23}\approx \epsilon^{2}$, the equations 
	\[\sum_{j=1}^{3}\Gamma_j w_j=\Gamma_{2}w_{12}+\Gamma_{3}w_{13}=\Gamma_{1}w_{21}+\Gamma_{3}w_{23}=\Gamma_{1}w_{31}+\Gamma_{2}w_{32} \prec \epsilon^2\]
	imply that 
	\begin{equation}\label{equ:triangle1}
		\frac{w_{12}}{\Gamma_{3}}\sim\frac{w_{23}}{\Gamma_{1}}\sim\frac{w_{31}}{\Gamma_{2}}\approx \epsilon^{2}.
	\end{equation}
	By the isolation of this triangle in the $z$-diagram, it holds that 
	\begin{equation}\label{equ:triangle2}
		\Lambda z_{1}\sim \frac{\Gamma_2}{w_{21}}+\frac{\Gamma_3}{w_{31}}, ~\Lambda z_{2}\sim \frac{\Gamma_1}{w_{12}}+\frac{\Gamma_3}{w_{32}},~\Lambda z_{3}\sim \frac{\Gamma_1}{w_{13}}+\frac{\Gamma_2}{w_{23}}.
	\end{equation}
	Since $z_1\sim z_2\sim z_3$, the equations \eqref{equ:triangle1} and \eqref{equ:triangle2} lead to 
	\[\frac{\Gamma_1}{\Gamma_{2}}-\frac{\Gamma_2}{\Gamma_{1}}=\frac{\Gamma_2}{\Gamma_{3}}-\frac{\Gamma_3}{\Gamma_{2}}=\frac{\Gamma_3}{\Gamma_{1}}-\frac{\Gamma_1}{\Gamma_{3}}.\] 
	This  contradicts with $\Gamma_{1}+\Gamma_{2}+\Gamma_{3}=0$. 
\end{proof}

Similarly, we have the following result.
\begin{proposition}\label{Prp:triangle2} 
	Suppose that a diagram has an isolated triangle of \(z\)-strokes in the \(z\)-diagram, where all three vertices, say \(\textbf{1}, \textbf{2}, \textbf{3}\), are \(z\)-circled. If \(z_1 \sim z_2 \sim z_3\), then there exists some \(k > 3\) such that \(z_{k1} \prec \epsilon^{-2}\). 
\end{proposition}

\begin{proof}
	Suppose that 	$z_1\sim z_2\sim z_3\approx \epsilon^{-2}$. By Proposition \ref{Estimate1} and Rule IV, we have $\Gamma_{1}+\Gamma_{2}+\Gamma_{3}=0$.
	Suppose that it holds  $z_{k1} \approx \epsilon^{-2}$ for all $k>3$. Then  $\frac{1}{z_{kj}}- \frac{1}{z_{k1}}=\frac{z_{1j}}{z_{kj}z_{k1}} \prec \epsilon^2$ for all $k>3, 1\le j\le 3$. Similar to the argument of the above result, we have  $\bar{\Lambda}\sum_{j=1}^{3}\Gamma_j w_j \prec \epsilon^{2}$, 
	\[  \frac{w_{12}}{\Gamma_{3}}\sim\frac{w_{23}}{\Gamma_{1}}\sim\frac{w_{31}}{\Gamma_{2}}\approx \epsilon^{2},\]
	\begin{equation}\notag
		\Lambda z_{1}\sim \frac{\Gamma_2}{w_{21}}+\frac{\Gamma_3}{w_{31}}, ~\Lambda z_{2}\sim \frac{\Gamma_1}{w_{12}}+\frac{\Gamma_3}{w_{32}},~\Lambda z_{3}\sim \frac{\Gamma_1}{w_{13}}+\frac{\Gamma_2}{w_{23}}, 
	\end{equation}
	and 
	$\frac{\Gamma_1}{\Gamma_{2}}-\frac{\Gamma_2}{\Gamma_{1}}=\frac{\Gamma_2}{\Gamma_{3}}-\frac{\Gamma_3}{\Gamma_{2}}=\frac{\Gamma_3}{\Gamma_{1}}-\frac{\Gamma_1}{\Gamma_{3}}.$
	This  contradicts with $\Gamma_{1}+\Gamma_{2}+\Gamma_{3}=0$. 
\end{proof}

\begin{proposition}\label{Prp:dumbbell} 
	Assume that vertices \textbf{1} and \textbf{2} are both $z$- and $w$-circled and connected by a $zw$-edge, and the sub-diagram formed by the two vertices  is isolated in the $z$-diagram. Assume that 
	vertices \textbf{3} and \textbf{4} are also both $z$- and $w$-circled and connected by a $zw$-edge,  and  is isolated in the $z$-diagram. 
	Then, there must exist some $k>4$ such that at least one among  $z_{k1}, w_{k1}, z_{k3}, w_{k3}$ is bounded (i.e., $\preceq 1$). 
\end{proposition}

\begin{proof}
	By Proposition \ref{Estimate1} and Rule IV, we have  \[z_1\sim z_2,  z_3\sim z_4, \ w_1\sim w_2, w_3 \sim w_4, \ \Gamma_{1}+\Gamma_{2}=0, \Gamma_3+\Gamma_{4}=0.\]
	Suppose that it holds that $w_{k1}\succ 1$ for all $k>4$. Then  $\frac{1}{w_{k2}}- \frac{1}{w_{k1}}=\frac{w_{12}}{w_{k2}w_{k1}} \prec \epsilon^2$ for all $k>4$.  
	Note that $z_{12}\approx \epsilon^2$, and 
	\[\Lambda z_{12}= (\Gamma_1+\Gamma_2)Z_{12}+ \Gamma_3( \frac{w_{21}}{w_{32}w_{31}} -\frac{w_{21}}{w_{42}w_{41}}) + \sum_{k>4}\Gamma_k(\frac{1}{w_{k2}}- \frac{1}{w_{k1}} ).     \]
	We conclude that  
	\[\frac{w_{21}}{w_{31}w_{32}}\approx  \frac{w_{21}}{w_{41}w_{42}}\succeq \epsilon^{2}\Rightarrow w_{31}\preceq 1\Rightarrow w_1\sim w_2\sim w_3\sim w_4.\]
	Similarly, we have\[z_1\sim z_2\sim z_3\sim z_4.\]

	Note that 
	\[\overline{\Lambda}\sum_{j=1}^{4}\Gamma_j w_j=\sum_{k> 4} \sum_{j=1}^{4}\frac{\Gamma_k\Gamma_j}{z_{kj}}=\sum_{k> 4} \Gamma_k \Gamma_1 (\frac{1}{z_{k1}}- \frac{1}{z_{k2}})  + \sum_{k> 4} \Gamma_k \Gamma_3 (\frac{1}{z_{k3}}- \frac{1}{z_{k4}})  \prec \epsilon^{2}. \]
	Then the equation $\Gamma_2 w_{12}+\Gamma_4 w_{34}=\sum_{j=1}^{4}\Gamma_j w_j$ leads to
	\[\Gamma_2 w_{12}\sim -\Gamma_4 w_{34}, \ \text{or} \ \Gamma_2 Z_{34}\sim -\Gamma_4 Z_{12} .\] 
	On the other hand, the isolation of the two segments implies 
	\[\Lambda z_2\sim \Gamma_1 Z_{12}, \ \Lambda z_4\sim \Gamma_3 Z_{34}, \Rightarrow \Gamma_1 Z_{12}\sim \Gamma_3 Z_{34}.\]
	As a result, we have
	\[\Gamma_1\Gamma_2=-\Gamma_3\Gamma_4, \ \text{or} \ \Gamma_1^2+\Gamma_3^2=0,\] 
	which is a contradiction.
	
\end{proof}

\begin{proposition}\label{Prp:quadrilateral} 
	Assume that there is a quadrilateral  with  vertices \textbf{1,  2, 3, 4}, that is fully $z$- and $w$-stroked, and fully  $w$-circled.   Moreover, the quadrilateral is isolated  in the $w$-diagram. 
	Then, there must exist some  $k>4$ such that $w_{k1} \preceq 1$. 
\end{proposition}

\begin{proof}
	We establish the result by contradiction. 
	Rule IV implies that   $\sum_{j=1}^{4}\Gamma_j=0.$ 	Suppose that it holds that $w_{k1}\succ 1$ for all $k>4$. Then $\frac{1}{w_{kj}}- \frac{1}{w_{k1}}\prec  \epsilon^{2}$ for $k>4, j\le4$, so
	\[\Lambda\sum_{j=1}^{4}\Gamma_j z_j=\sum_{k>4}\Gamma_k \sum_{j=1}^{4}\frac{\Gamma_j}{w_{kj}}=\sum_{k>4}\Gamma_k  \sum_{j=1}^{4} \Gamma_j  (\frac{1}{w_{k1}}+\frac{1}{w_{kj}}- \frac{1}{w_{k1}} )   \prec \epsilon^{2}.\]
	Then  \[  \epsilon^2 \succ \sum_{j=1}^{4}\Gamma_j z_j=\sum_{j=2, 3, 4}\Gamma_j z_{1j}\Rightarrow \Gamma_2 z_{12}\sim -\Gamma_3 z_{13}-\Gamma_4 z_{14}\approx \epsilon^{2}.  \]
	Set $z_{13}\sim a \epsilon^2, z_{14}\sim b \epsilon^2$, where $a\ne b$ are some nonzero constants. Then 
	\begin{align*}
		&z_{12}\sim -\frac{\Gamma_3 a+\Gamma_4 b}{\Gamma_2} \epsilon^2,    & z_{23}\sim \frac{a (\Gamma_2+\Gamma_3)+\Gamma_4 b}{\Gamma_2} \epsilon^2,\\
		&z_{24}\sim \frac{\Gamma_3 a+b (\Gamma_2+\Gamma_4)}{\Gamma_2} \epsilon^2,    &z_{34}\sim (b-a) \epsilon^2. 
	\end{align*}
	Since $w_1\sim w_2\sim w_3\sim w_4$, we set $\bar{\Lambda}w_k\sim \frac{1}{c\epsilon^2}, k=1, 2, 3, 4$.
	Substituting those into the system
	\[\bar{\Lambda}w_k\sim \sum_{j\ne k, j=1}^{4}\frac{\Gamma_j}{z_{jk}}, \ k=1, 2,3,4,\]
	which is from the isolation of the quadrilateral in $w$-diagram. 
	We obtain four homogeneous  polynomials of the three variables $a, b, c$. Thus, we set $c=1$, and obtain the following four polynomials of the five variables $a,b,  \Gamma_1, \Gamma_{3}, \Gamma_{4}$, 
	\begin{align*}
		&a^2 (-\Gamma_3) (b+\Gamma_4)+a b \left(-\Gamma_4 (b-2 \Gamma_3)+\Gamma_1^2+2 \Gamma_1 (\Gamma_3+\Gamma_4)\right)-b^2 \Gamma_3 \Gamma_4=0,\\
		&a^3 \Gamma_3^2 (\Gamma_1+\Gamma_4)+a^2 \Gamma_3 \left(-b \left(\Gamma_1^2+\Gamma_1 \Gamma_3+\Gamma_4 (2 \Gamma_3-\Gamma_4)\right)-(\Gamma_1-\Gamma_3+\Gamma_4) (\Gamma_1+\Gamma_3+\Gamma_4)^2\right)\\
		&-a b \left(\Gamma_1^2 \Gamma_4 (b-4 \Gamma_3)+\Gamma_1 \left(\Gamma_4^2 (b+2 \Gamma_4)+2 \Gamma_3^3-2 \Gamma_3^2 \Gamma_4-2 \Gamma_3 \Gamma_4^2\right)-\Gamma_3^2 \Gamma_4 (b+2 \Gamma_4)\right)\\
		&-a b \left(2 b \Gamma_3 \Gamma_4^2-\Gamma_1^4-2 \Gamma_1^3 (\Gamma_3+\Gamma_4)+\Gamma_3^4+\Gamma_4^4\right)\\
		&+b^2 \Gamma_4 \left(\Gamma_1 \left(\Gamma_4 (b+\Gamma_4)-3 \Gamma_3^2-2 \Gamma_3 \Gamma_4\right)+\Gamma_3 \Gamma_4 (b+\Gamma_4)-\Gamma_1^3-\Gamma_1^2 (3 \Gamma_3+\Gamma_4)-\Gamma_3^3-\Gamma_3^2 \Gamma_4+\Gamma_4^3\right)=0,\\
		&a^3 (\Gamma_1+\Gamma_4)+a^2 (\Gamma_3 (2 \Gamma_1+\Gamma_3+2 \Gamma_4)-b (\Gamma_1+2 \Gamma_4))+a b \left(b \Gamma_4-2 \Gamma_1 \Gamma_3-\Gamma_3^2-2 \Gamma_3 \Gamma_4\right)-b^2 \Gamma_1 \Gamma_4=0,\\
		&a^2 \Gamma_3 (b-\Gamma_1)-a b (b (\Gamma_1+2 \Gamma_3)+\Gamma_4 (2 \Gamma_1+2 \Gamma_3+\Gamma_4))+b^2 (b (\Gamma_1+\Gamma_3)+\Gamma_4 (2 \Gamma_1+2 \Gamma_3+\Gamma_4))=0. 
	\end{align*}
	Tedious but standard computation, such as calculating the Gr\"{o}bner basis,   yields 
	\[  b^5 (\Gamma_1+\Gamma_3+\Gamma_4) \left(\Gamma_1^2+\Gamma_1 \Gamma_3+\Gamma_1 \Gamma_4+\Gamma_3^2+\Gamma_3 \Gamma_4+\Gamma_4^2\right)=0. \]
	It is a contradiction since $b\ne 0, \Gamma_1+\Gamma_3+\Gamma_4=-\Gamma_2\ne 0$ and 
	\[\Gamma_1^2+\Gamma_1 \Gamma_3+\Gamma_1 \Gamma_4+\Gamma_3^2+\Gamma_3 \Gamma_4+\Gamma_4^2=\frac{1}{2}( \Gamma_1^2+\Gamma_2^2+\Gamma_3^2+\Gamma_4^2)\ne 0.\]
\end{proof}

\section{Algorithm for the two-colored diagrams}\label{sec:matrixrules}

\indent\par   
The initial step of the singular sequence method involves identifying all two-colored diagrams introduced in Section \ref{sec:pri}. These diagrams represent potential scenarios where finiteness may fail. We successfully determined all such diagrams for the five-vortex problem  using the approach outlined by Albouy and Kaloshin \cite{Albouy2012Finiteness}.   However, this method demands meticulous attention and time,  particularly when applying the singular sequence method to address the finiteness issue for $N\ge 6$. In response to this, inspired by the work of Chang and Chen, we developed an algorithm based on symbolic computations to determine the diagrams.

Building on Chang and Chen's groundwork \cite{2023Chen-1, 2023Chen-2, 2023Chen-3},  we  introduce notation of the $zw$-matrices and transform Rule I-VI from Section \ref{sec:rule}   into matrix rules for the  $zw$-matrices. While there are 18 matrix rules for  the N-body problem,   we have 11  rules for the eliminating of $zw$-matrices for the N-vortex problem, reflecting the fact that we have less rules for the two-colored diagrams than that of the N-body problem. Among them, the Rule of Column Sums, Rule of Trace, Rule of Circling, Rule of Connected Components and Rule of Trace-0 Principal Minors mirror those of the N-body problem. Meanwhile, the two Rules of  Triangle, Rule of Trace-2 Matrices,  Rule of Fully Edged Components, Rule of Quadrilateral and Rule of Dumbbells are specific to  the N-vortex problem.  Finally, we outline an algorithm designed to determine these diagrams.

\subsection{Matrix rules}
\begin{defn}
	Given a singular sequence of complex normalized central configurations for the $n$-vortex
	problem and consider the corresponding $zw$-diagram. 
	The matrix representing the adjacency of its $z$-diagram, where $z$-circles are treated as self-loops, is termed the associated $z$-\emph{matrix} for the singular sequence. To elaborate, the $z$-matrix is defined as a symmetric $N\times N$  matrix, denoted by $A=\left(a_{ij}\right)$, where
	\[
	\begin{aligned}a_{ij} & =1\text{ if }i\neq j\text{ and there is a }z\text{-stroke between vorticities }\Gamma_{i}\text{ and }\Gamma_{j};\\
	& =0\text{ if }i\neq j\text{ and there is no }z\text{-stroke between vorticities }\Gamma_{i}\text{ and }\Gamma_{j};\\
	a_{ii} & =1\text{ if there is a }z\text{-circle at }\Gamma_{i};\\
	& =0\text{ if there is no }z\text{-circle at }\Gamma_{i}.
	\end{aligned}
	\]
	Similarly, the $w$-{\emph{matrix}}
	associated to the singular sequence is the adjacency matrix for the
	$w$-diagram with $w$-circles regarded as self-loops. It is  a symmetric $N\times N$ matrix, denoted by $B=\left(b_{ij}\right)$, where
	\[
	\begin{aligned}b_{ij} & =1\text{ if }i\neq j\text{ and there is a }w\text{-stroke between vorticities }\Gamma_{i}\text{ and }\Gamma_{j}\\
	& =0\text{ if }i\neq j\text{ and there is no }w\text{-stroke between vorticities }\Gamma_{i}\text{ and }\Gamma_{j};\\
	b_{ii} & =1\text{ if there is a }w\text{-circle at }\Gamma_{i;}\\
	& =0\text{ if there is no }w\text{-circle at }\Gamma_{i}.
	\end{aligned}
	\]
	The $zw$-{\emph{matrix}} of the singular
	sequence is the $N\times2N$  matrix $(A\mid B)$. 
\end{defn}

\begin{defn}
	A \emph{connected component} of a $z$-matrix $A$ (resp. $w$-matrix
	$B$ ) is a subset $I$ of $\{1,2,\ldots,N\}$ such that the principal
	submatrix obtained by removing all $j$-th columns and all $j$-th
	rows with $j\notin I$ from $E+A+A^{2}+...+A^{N-1}$ (resp. $E+B+B^{2}+...+B^{N-1}$
	) is a submatrix with all positive entries.  Here,  $E$ is the identity
	matrix of order $N$.
\end{defn}

In what follows we assume symmetric
$N\times N$ binary matrices $A=\left(a_{ij}\right)$ and $B=\left(b_{ij}\right)$
are respectively the $z$-matrix and $w$-matrix associated to a singular
sequence of complex normalized central configurations. The rules for  two-colored diagrams can be expressed as  the following  matrix rules. 

\begin{theorem}
	Let $M=A$ or $B$. 
\end{theorem}
\begin{enumerate}
	\item \textbf{(Rule of Column Sums)} No column sum of $M$ is equal to 1,
	and column sums of $M$ cannot be all equal to 0 . 
	\item \textbf{(Rule of Trace) } The trace $\text{Tr}(M)$ of $M$
	cannot be 1 .
\end{enumerate}
\begin{proof}
	The first statement follows  from Rule I , and the second one
	from Rule III. 
\end{proof}
\begin{theorem}
	\textbf{(First Rule of Triangle) }  Let  $N\geq3$.  
	It holds that 
	\[
	a_{ij}+a_{jk}+a_{ki}\neq 2, \quad 1\leq i<j<k\leq N
	\]
	It holds if $A$ is replaced by $B$.
\end{theorem}
\begin{proof}
	It follows  from Rule VI. 
\end{proof}

\begin{theorem}
	\textbf{(Rule of Circling) } For any $1\leq i<j\leq N$, we have 
	\[
	a_{ij}\left(b_{ii}+b_{jj}\right),b_{ij}\left(a_{ii}+a_{jj}\right)\neq1.
	\]
\end{theorem}
\begin{proof}
	If $a_{ij}=1$, then vertex \textbf{i }and \textbf{j} are $w$-close
	by the estimate of Proposition \ref{Estimate1}. So $b_{ii},b_{jj}$ are both zero or one by Rule
	II. Then the statement follows. 
\end{proof}
\begin{theorem}
	\textbf{(Rule of Trace-2 Matrices) }
\end{theorem}
\begin{enumerate}
	\item If $\text{Tr}(A)=2$ and $a_{ii}=a_{jj}=1$ for some $i\ne j$,
	then $a_{ij}b_{ij}=0.$
	\item If $\text{Tr}(B)=2$ and $b_{ii}=b_{jj}=1$ for some $i\ne j$,
	then $a_{ij}b_{ij}=0.$
\end{enumerate}
\begin{proof}
	It is enough to prove the first statement. Since $\text{Tr}(A)=2$
	and $a_{ii}=a_{jj}=1$, no other vertex is $z$-close to the two
	vertices \textbf{i},\textbf{j}. If $b_{ij}=1,$$z_{ij}\text{\ensuremath{\approx\epsilon^{2}}}$
	by the estimate of Proposition \ref{Estimate1}, then the two vertices are also $z$-close. By Proposition
	\ref{Prp:sumT12}, $a_{ij}=0$, and the proof is completed. 
\end{proof}
\begin{theorem}
	\textbf{(Rule of Connected Components)} If $I$ is a connected component of $A$  (resp. $B$), then 
	$\sum_{i\in I}a_{ii}\neq1$ (resp. $\sum_{i\in I}b_{ii}\neq1$ ),
\end{theorem}
\begin{proof}
	If $I$ is a connected component of $A$, then
	$I$ is an isolated component in the $z$-diagram. Then $\sum_{i\in I}a_{ii}\neq1$
	by Rule IV. 
\end{proof}

\begin{theorem}
	\textbf{(Rule of Trace-0 Principal Minors)} Let $I\subset\{1,\ldots,N\}$.
	If $a_{ii}=0$ for all $i\in I$, then $\sum_{i\in I,j\notin I}a_{ij}\neq1$.
	This rule holds if $A$ is replaced by  $B$.
\end{theorem}
\begin{proof} The proof is  same as that in \cite{2023Chen-1}. For the completeness, we reproduce it here. 
	Assume that $\sum_{i\in I,j\notin I}a_{ij}=1$. In this case, there exists a unique pair $\left(i_{0},j_{0}\right)$ with $i_{0}\in I$ and $j_{0}\notin I$ such that $a_{i_{0}j_{0}}=1$. By examining system \eqref{equ:complexcc}, we find $\Lambda \sum_{i\in I}\Gamma_{i}z_{i}=\sum_{i\in I,j\notin I}\Gamma_{i}\Gamma_{j}Z_{ij}.$ However, as $a_{ii}=0$ for all $i\in I$, the term $\Gamma_{i_{0}}\Gamma_{j_{0}}Z_{i_{0}j_{0}}$ is the unique term of maximal order in this equation, leading to a contradiction. Thus, it is impossible for $\sum_{i\in I,j\notin I}a_{ij}$ to equal 1.
\end{proof}
The Rule of Trace-0 Principal Minors can be skipped
when $N\leq5$. 
It becomes  useful for $N\ge 6$.

The next three Rules  follow easily from Corollary \ref{Cor:L&Gamma}, Proposition \ref{Prp:triangle}, Proposition \ref{Prp:quadrilateral}  and Proposition \ref{Prp:dumbbell} respectively.

\begin{theorem}
	\textbf{(Rule of Fully Edged Components)} Assume that $K$ is a connected
	component of $A$, $\sum_{i\in K}1\ge3$, $a_{ij}=1$ for all $i\ne j$
	in $K$ and $\sum_{i\in K}a_{ii}=0$. 
\end{theorem}
\begin{enumerate}
	\item If $I$ is a connected component of $B$, and $K\subset I,$ then it is
	impossible that $\sum_{i\in I}b_{ii}=\sum_{i\in K}b_{ii}=\sum_{i\in K}1$.
	\item For any subset \(K_{1}\), obtained by deleting one element from \(K\), if \(I\) is a connected component of \(B\) and \(K_{1} \subset I\), then it is impossible that \(\sum_{i \in I} b_{ii} = \sum_{i \in K_{1}} b_{ii} = \sum_{i \in K_{1}} 1\).
	\item   If \(J\) is a connected component of \(B\), with \(K \subset J\) and \(\card (K) + 1 = \card(J)\), then it is impossible that \(b_{ij} = 1\) for all \(i \ne j\) in \(J\) and \(\sum_{i \in J} b_{ii} = 0\).
	\item If $\{ I_j, j=1, ..., s\}$ are  $s$ connected components of $B$, and $K\subset \cup_{j=1}^s I_j,$ then it is
	impossible that $\sum_{i\in \cup_{j=1}^s I_j }b_{ii}=\sum_{i\in K}b_{ii}=\sum_{i\in K}1$.
\end{enumerate}
They hold if we switch $A$ and $B$.

\begin{theorem} \textbf{(Second Rule of Triangle)}
	Assume that $K$ is a connected
	component of $A$, $\sum_{i\in K}1=3$, $\sum_{i\in K}a_{ii}  \sum_{i\ne j, i, j\in K}a_{ij}=9$ and $\sum_{i\in K}b_{ii}  \sum_{i\ne j, i, j\in K}b_{ij}=9.$ Then  $\text{Tr}(A)>3$. 
	It holds  if $A$ is replaced by  $B$.
\end{theorem}

\begin{theorem} \textbf{(Rule of  Quadrilateral)}
	Assume that $K$ is a connected
	component of $A$, $\sum_{i\in K}1=4$, $\sum_{i\in K}a_{ii}  \sum_{i\ne j, i, j\in K}a_{ij}=24$ and $ \sum_{i\ne j, i, j\in K}b_{ij}=6.$ Then  $\text{Tr}(A)>4$. 
	It holds  if $A$ is replaced by  $B$.
\end{theorem}

\begin{theorem} \textbf{(Rule of Dumbbells)}
	Assume that $K_1, K_2$ are two  connected
	components of $A$, $\sum_{i\in K_s}1=2$, $\sum_{i\in K_s}a_{ii}  \sum_{i\ne j, i, j\in K_s}a_{ij}=4$ and $\sum_{i\in K_s}b_{ii}  \sum_{i\ne j, i, j\in K_s}b_{ij}=4$ for $s=1, 2$.  Then  $\text{Tr}(A)+ \text{Tr}(B)>8$. 
	It holds  if $A$ is replaced by  $B$.
\end{theorem}

\subsection{Algorithm for determining $zw$-diagrams}\label{sec:algorithm1}
The following is the outline of the algorithm to find all possible diagrams. We refer the readers to \cite{2023YuZhu} for  codes  with Mathematica. 
\begin{enumerate}
	\item Initial selection of possible $z$-matrices
	\begin{enumerate}
		\item Sort symmetric binary matrices based on  column sums and diagonals
		\item Eliminate matrices with trace 1
		\item Classify matrices according to trace
	\end{enumerate}
	\item Apply three of the matrix rules
	\begin{enumerate}
		\item Rule of Column Sums
		\item First Rule of Triangles 
		\item Rule of Trace-0 Principal Minors
	\end{enumerate}
	\item Produce possible $zw$-matrices
	\begin{enumerate}
		\item Produce possible $w$-matrices 
		\item Match  $z$-and $w$-matrices 
	\end{enumerate}
	\item Apply the other seven matrix rules
	\begin{enumerate}
		\item Rule of Circling
		\item Rule of Trace-2 Matrices
		\item Rule of Connected Components 
		\item Rule of Fully Edged Components
		\item Second Rule of Triangle
		\item Rule of Quadrilateral 
		\item Rule of Dumbbells
	\end{enumerate}
	\item Generate $zw$-diagrams associated to remaining $zw$-matrices
\end{enumerate}

In below we discuss  details of the program.

\emph{Step (1)}: 
	We begin by selecting all potential $z$-matrices from the set of  $N\times N$ symmetric binary matrices.  By the Rule of Trace, we first exclude  matrices with  trace  $1$. To  reduce equivalent cases, we assume, without loss of generality, that $z$-matrices adhere to the condition 
\begin{equation}
	a_{ii}\leq a_{jj},\quad\forall i<j.\label{eq:algorithm}
\end{equation}
These matrices are classified into $N$ classes, denoted as $S_{1},\ldots,S_{N}$, based on their traces. Here, $S_{1}$ consists of  matrices with trace  $0$, and $S_{k}$ includes matrices with  trace  $k$ for $k>1$.  The cardinality of each $S_k$ is $2^{N(N-1)/2}$, so the 
total number of  $z$-matrices is $N2^{N(N-1)/2}$.

Before applying matrix rules, we further decrease equivalent cases by sorting vertices to ensure increasing column sums, eliminating duplicates by considering equivalence under permutations of vertices,   and transforming them to meet property \eqref{eq:algorithm}. 

  For  $N=5$,  the cardinalities of the five subsets of $z$-matrices at the end of this step  are 
	\[\card(S_1)=84, \  \card(S_2)=240, \ \card(S_3)=240, \ \card(S_4)=182, \ \card(S_5)=84.\]

\emph{Step (2)}: From the $z$-matrices obtained in Step (1), we remove matrices that fail to satisfy the Rule of Column Sums, followed by those not meeting the First Rule of Triangles, and finally those not satisfying the Rule of Trace-0 Principal Minors. We then eliminate duplicates, accounting for equivalence under vertex permutations. 

  For  $N=5$,  the cardinalities of the five subsets of $z$-matrices at the end of this  step  are reduced  to: 
	\[\card(S_1)=3, \  \card(S_2)=5, \ \card(S_3)=4, \ \card(S_4)=4, \ \card(S_5)=2.\]

\emph{Step (3)}.  We first generate   potential  $w$-matrices. Similar to  $z$-matrices,   we categorize $w$-matrices into sets $T_{k}, 1\le k\le N$  based on their traces, where $T_1$ contains matrices with trace 0, 
	and $T_{k}$ contains matrices with trace  $k$  for $2\leq k\leq N$.  Although the initial rules applied to 
	$z$-matrices also apply to 
	$w$-matrices,  
we cannot eliminate $w$-matrices that are
related by conjugation of permutation matrices, since after matching
with a given $z$-matrix they are generally non-equivalent. 
	Thus, each set $T_k$
	includes all matrices obtained by conjugating every element of 
	$S_k$
	with all possible permutation matrices.

 For example,  when $N=5$, the cardinalities of the five subsets of $w$-matrices are
	\[\card(T_1)=16, \  \card(T_2)=90, \ \card(T_3)=70, \ \card(T_4)=55, \ \card(T_5)=11.\]

To decrease
the initial number of $zw$-matrices, we match each $z$-matrix only
with appropriate classes of $w$-matrices as follows. 
We assume that  each $zw$-matrix $(A\mid B)$ satisfies $$\text{Tr}(A)\leq\text{Tr}(B).$$
The matching process occurs between matrices in $S_{k}$ and matrices in $T_{k}\cup\cdots\cup T_{N}$ for $1\leq k\leq N$. The resulting sets of $zw$-matrices are denoted as $U_{k}$.

  For $N=5$,  the cardinalities of the five subsets of $zw$-matrices  at the end of this  step are
	\[\card(U_1)=726, \  \card(U_2)=1130, \ \card(U_3)=544, \ \card(U_4)=264, \ \card(U_5)=22.\]

\emph{Step (4)}.  For those $zw$-matrices obtained in Step (3), we delete those matrices that do not satisfy the seven matrix rules of Step (4).  
The order of applying them significantly influences algorithm efficiency.  Some rules, like the Rule of Circling, involve fewer computations but exclude more matrices, and it is natural to implement them first. We apply the seven rules in the order  they are listed in the outline, and then eliminate duplicates  up to permutations of vertices.  

To enhance efficiency, we divide the Rule of Circling into two parts: the first part
$$b_{ij}\left(a_{ii}+a_{jj}\right)\neq1$$
applies to matrices in $U_{K}$ with $2\leq K\leq N-1$ since $b_{ij}\left(a_{ii}+a_{jj}\right)\neq1$ is false only when $1\leq\text{Tr}(A)\leq N-1$.
Since each $zw$-matrix $(A\mid B)$ satisfies \eqref{eq:algorithm}, 
in $U_{K},2\leq K\leq N-1$, the discriminant $b_{ij}\left(a_{ii}+a_{jj}\right)\neq1$
becomes $b_{ij}\neq1$ for $i<N-K$ and $j\geq N-K$.  The second part ($a_{ij}\left(b_{ii}+b_{jj}\right)\neq1$)  is true for $U_{N}.$

The first part of the Rule of Trace-2 Matrices applies to matrices
in $U_{2}.$ The property \eqref{eq:algorithm} 
implies that $a_{NN}=a_{(N-1)(N-1)}=1$ for every $(A\mid B)\in U_{2}$.
Therefore,  the discriminant is reduced to $a_{N(N-1)}b_{N(N-1)}=0$.
We can skip the second part. Note that it only applies to matrices
of the form $(A\mid B)$ with $\text{Tr}(A)=0,2$ and $\text{Tr}(B)=2$.
Assume that $b_{11}=b_{22}=b_{12}=a_{12}=1$. Then both vertices \textbf{1}
and \textbf{2} are $z$-circled, i.e., $a_{11}=a_{22}=1$ and it reduces
to the first case. Otherwise, Rule I of two-colored diagrams implies that there must be other
$z$-stroke connected with vertices \textbf{1} or \textbf{2}, and
then by the Rule of Circling, there is other $w$-circled vertex.
This contradicts with $\text{Tr}(B)=2$.

For the Rule of Connected Components, it suffices to check components with cardinality at least three . If there are only one or two vertices in the component, then $\sum a_{ii}\ne1$ by the Rule of Column Sums. 

  For  $N=5$,  the cardinality of possible $zw$-matrices  at the end of this  step is  reduced to 31.

\emph{Step (5)}.  The remaining $zw$-matrices correspond to all possible two-colored diagrams. For each matrix,  we construct the corresponding two-colored diagram by the manner described in Sect. \ref{sec:rule}.

\section{Example: $N=5$}
\label{sec:diagram&constraints}

\indent\par 
For $N=5$,  the $zw$-matrices algorithm gives 31 diagrams.  We divide them into two lists, Figure \ref{fig:list1} and \ref{fig:list2}.  Notably, the first list of nine diagrams can be 
excluded by some further arguments. We will report this in a future work.

\begin{figure}[!h]

	\centering
	\includegraphics[width=0.9\textwidth]{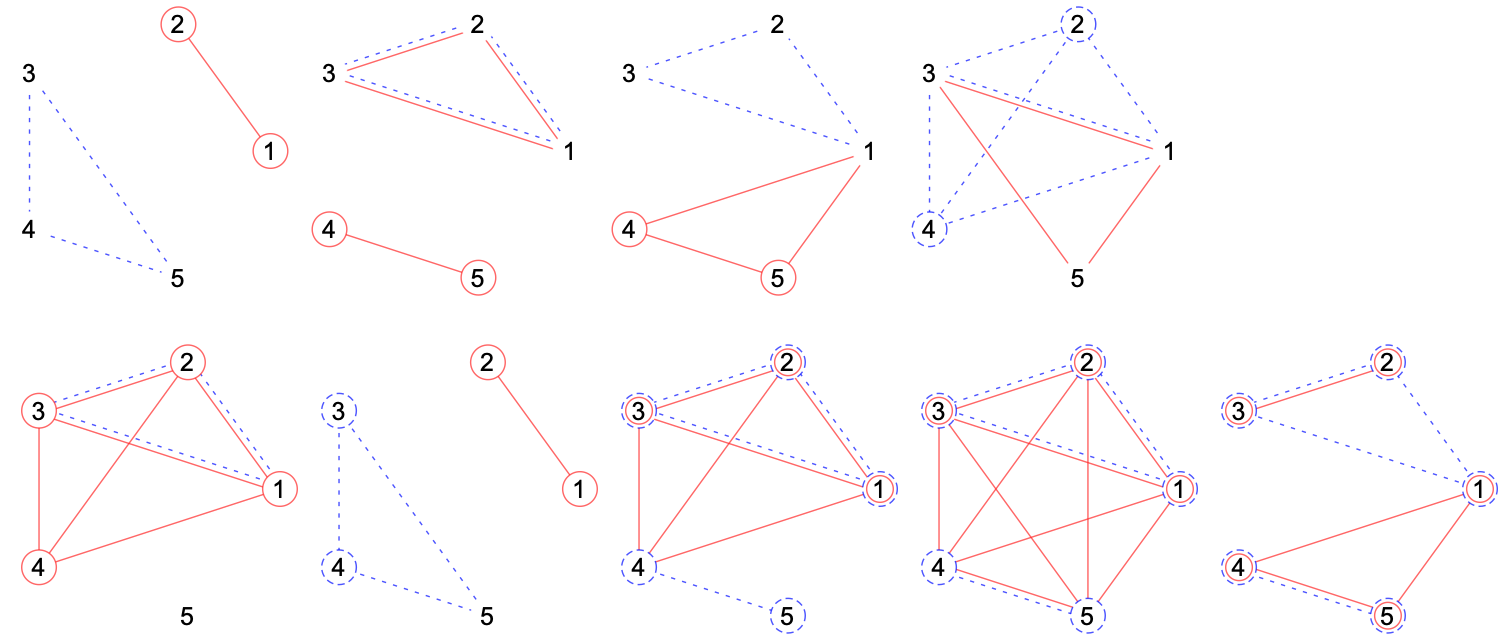}

	\caption{The nine  impossible diagrams.}
	\label{fig:list1}
\end{figure}

\begin{figure}[!h]

	\centering
	\includegraphics[width=0.9\textwidth]{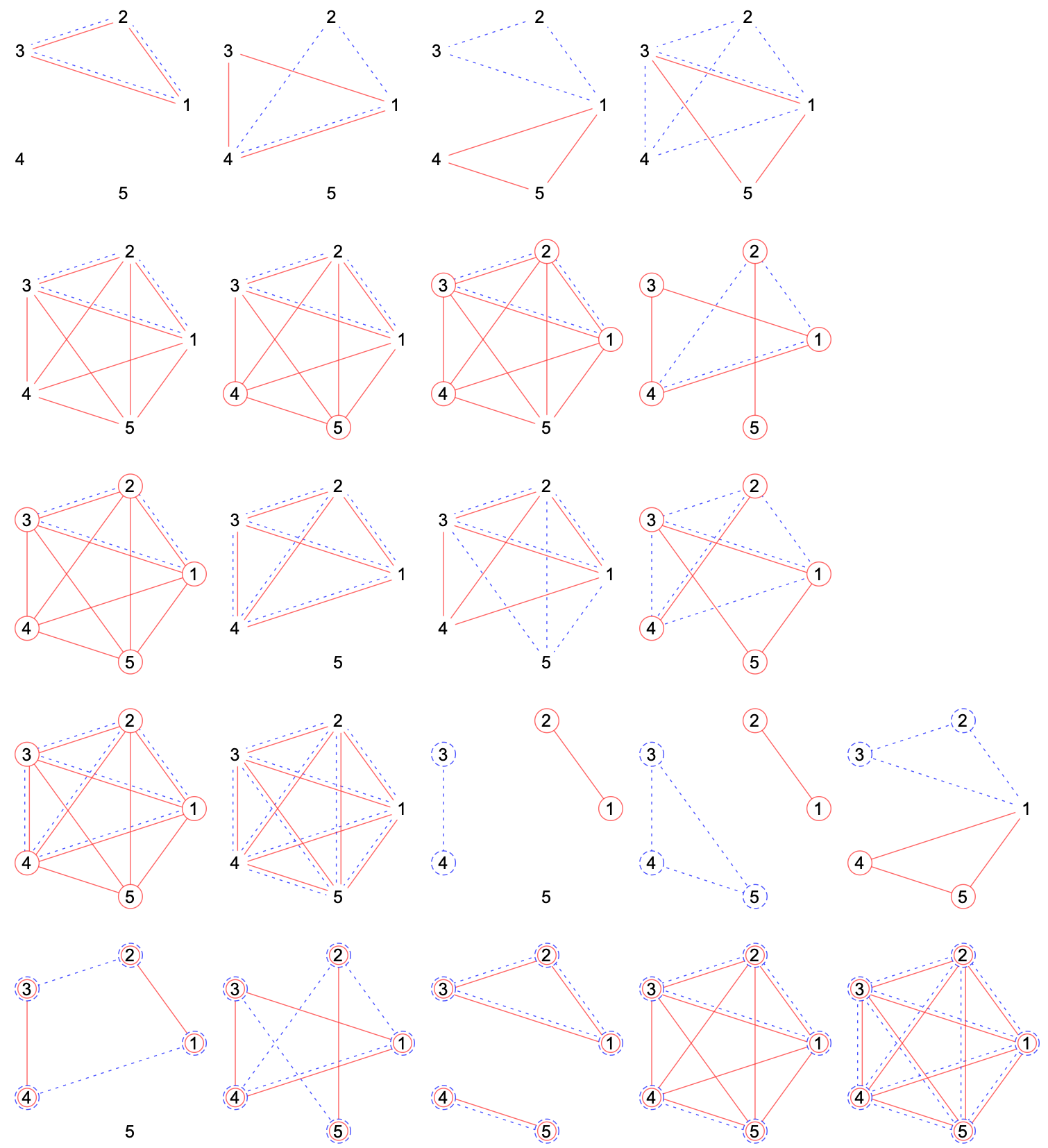}

	\caption{The 22 possible diagrams.  }
	\label{fig:list2}
\end{figure}

\end{document}